\documentclass[11pt,leqno]{amsart}
\usepackage{amsmath,amsthm,amssymb}
\usepackage{tabularx}
\newcommand{\GF}{{\mathbb F}}
\newcommand{\FF}{{\mathbb F}}
\newcommand{\CC}{{\mathbb C}}
\newcommand{\Z}{{\mathbb Z}}
\newcommand{\ZZ}{{\mathbb Z}}
\newcommand{\R}{{\mathbb R}}
\newcommand{\RR}{{\mathbb R}}

\newcommand{\cL}{\mathcal{L}}

\newcommand{\wt}{{\rm wt}}
\newcommand{\supp}{{\rm supp}}

\DeclareMathOperator{\Harm}{Harm}

\usepackage{color} 
\usepackage{ulem}

\newtheorem{Thm}{Theorem}[section]
\newtheorem{thm}{Theorem}[section]
\newtheorem{Lem}[Thm]{Lemma}
\newtheorem{lem}[Thm]{Lemma}

\newtheorem{cor}[thm]{Corollary}

\newtheorem{prop}[Thm]{Proposition}

\theoremstyle{definition}
\newtheorem{Def}[Thm]{Definition}
\newtheorem{df}[Thm]{Definition}

\newtheorem{rem}[Thm]{Remark}
\newtheorem{Ex}[Thm]{Example}

\numberwithin{equation}{section}
\title{A note on Assmus--Mattson type theorems}

\author{Tsuyoshi Miezaki*}
\thanks{*Corresponding author}
\address{Faculty of Education, University of the Ryukyus, Okinawa  
903--0213, Japan}
\email{miezaki@edu.u-ryukyu.ac.jp}
\author{Akihiro Munemasa}
\address{Graduate School of Information Sciences, Tohoku University, Sendai
980--8579, Japan}
\email{munemasa@math.is.tohoku.ac.jp}
\author{Hiroyuki Nakasora}
\address{Institute for Promotion of Higher Education, Kobe Gakuin University, Kobe
651--2180, Japan}
\email{nakasora@ge.kobegakuin.ac.jp}

\date{February 25, 2020}



\begin{document}
\begin{abstract}
In the present paper, 
we give Assmus--Mattson type theorems for 
codes and lattices. 
We show that 
a {binary} doubly even self-dual code 
of length $24m$ 
with minimum weight $4m$ 
provides a combinatorial $1$-design and
an even unimodular lattice of rank $24m$ with minimum norm $2m$ 
provides a spherical $3$-design. We remark that 
some of such codes and lattices give $t$-designs for 
higher $t$. 
As a corollary, 
we give some restrictions on the 
weight enumerators of {binary} doubly even self-dual codes 
of length $24m$ with minimum weight $4m$. 
Ternary and quaternary analogues are also given. 
\end{abstract}
\maketitle

%%%%%%%%%%%%%%%%%%%%%%%%%%%%%%%%%%%%%%%%%%%%%%%%%%%%%%%%%%%%%%%%%%%%%%%%%%%%%%

\section{Introduction}

Let $C$ be a code over $\FF_q$ and 
$C_\ell:=\{c\in C\mid \wt(c)=\ell\}$. 
%(We remark that $C_\ell$ is not a multiset). 
Let $L \subset \R^{n}$ be a lattice and 
$L_\ell:=\{x\in L\mid (x,x)=\ell\}$. 
In this paper, we call $C_\ell$ (resp.\ $L_\ell$) a shell of the code $C$ 
(resp.\ lattice $L$)  whenever it is non-empty.

Shells of extremal self-dual codes are known to support combinatorial designs
by the Assmus--Mattson theorem.
More precisely, the set $\mathcal{B}(C_\ell):=\{\supp(x)\mid x\in C_\ell\}$
forms the set of blocks of a combinatorial design. Note that
$\mathcal{B}(C_\ell)$ is not a multiset, so combinatorial designs 
in this papers
are always without repeated blocks.
Similarly, shells of extremal even unimodular lattices are known to give
spherical designs by a theorem of Venkov.
%Shells of extremal self-dual codes are known to support combinatorial designs,
%and those of extremal even unimodular lattices are known to give
%spherical designs. 
%{\color{red}
%(In this paper, we assume that $C_\ell$ is 
%the support (with duplicates omitted) design. 
%In other words, the set of blocks is not a multiset.)
%}
%The former can be regarded as a variant of the Assmus--Mattson theorem,
%while the latter is a theorem of Venkov.
In the present paper, 
we give analogues of these theorems with relaxed assumptions; the minimum weight
of a self-dual code is allowed to be slightly smaller than the extremal case,
the minimum norm of an even unimodular lattice is allowed to be smaller by $2$
than the extremal case. The conclusions we obtain are necessarily weaker than
the extremal cases, but still they are nontrivial, and give a restriction
on weight enumerators for the binary case.

To explain our results, 
we review some results on codes and lattices. 
Let $C$ be a {binary} doubly even self-dual code of length $n$. 
Then we have the following bound on the minimum weight of 
$C$ \cite{MS}:
\begin{align}\label{ine:code-ext}
\min(C)\leq 4\left\lfloor\frac{n}{24}\right\rfloor+4. 
\end{align}
We say that $C$ meeting
the bound (\ref{ine:code-ext}) with equality is extremal. 
Let $C$ be an extremal code of length $n$. 
%{\color{blue} TO BE DELETED: , and 
%let us set 
%\[
%C_\ell:=\{c\in C\mid \wt(c)=\ell\}. 
%\]
%}
%Then %by Assmus--Mattson theorem, 
%any 
Identifying codewords with their support, 
$C_\ell$
forms a combinatorial $t$-design, where
\[
t=\begin{cases}
5 &\text{if $n \equiv 0 \pmod{24}$},\\
3 &\text{if $n \equiv 8 \pmod{24}$},\\
1 &\text{if $n \equiv 16 \pmod{24}$},
\end{cases}
\]
provided $C_\ell\neq \emptyset$ \cite{assmus-mattson}. 
Ternary and quaternary analogues of this fact 
were also given in \cite{assmus-mattson}. 
Let $C$ be a ternary or quaternary self-dual code of length $n$. 
Then we have the following bound on the minimum weight of 
$C$ \cite{MOSW,MS}:
\begin{equation}\label{34}
\min(C)\leq 
\begin{cases}
3\left\lfloor\frac{n}{12}\right\rfloor+3
&\text{if $C$ is ternary,}\\
2\left\lfloor\frac{n}{6}\right\rfloor+2
&\text{if $C$ is quaternary.}
\end{cases}
\end{equation}
We say that $C$ meeting
the bound (\ref{34}) with equality is extremal. 
Let $C$ be an extremal code of length $n$
and 
$w$ be the largest integer satisfying
\[
w-\left\lfloor\frac{w+q-2}{q-1}\right\rfloor<d. 
\]
Then 
for $\ell\leq w$, 
$C_\ell$
forms a combinatorial $t$-design, where
\[t=\begin{cases}
5 &\text{if $n \equiv 0 \pmod{12}$,}\\
3 &\text{if $n \equiv 4 \pmod{12}$,}\\
1 &\text{if $n \equiv 8 \pmod{12}$,}
\end{cases}\]
if $C$ is ternary, and
\[t=\begin{cases}
5 &\text{if $n \equiv 0 \pmod{6}$,}\\
3 &\text{if $n \equiv 2 \pmod{6}$,}\\
1 &\text{if $n \equiv 4 \pmod{6}$,}
\end{cases}
\]
if $C$ is quaternary, 
provided $C_\ell\neq \emptyset$ \cite{assmus-mattson}.

Let $L$ be an even unimodular lattices of rank $n$. 
Then we have the following bound on the minimum norm of 
$L$ \cite{MOS75}:
\begin{align}\label{ine:lattice-ext}
\min(L)\leq 2\left\lfloor\frac{n}{24}\right\rfloor+2. 
\end{align}
We say that $L$ meeting
the bound (\ref{ine:lattice-ext}) with equality is extremal. 
Let $L$ be an extremal lattice of length $n$. 
%{\color{blue} TO BE DELETED: 
%, and 
%let us set 
%\[
%L_\ell:=\{x\in L\mid (x,x)=\ell\}. 
%\]
%}
After normalization,
$L_\ell$
forms a spherical $t$-design, where
\[t=
\begin{cases}
11 &\text{if $n \equiv 0 \pmod{24}$,}\\
7  &\text{if $n \equiv 8 \pmod{24}$,}\\
3  &\text{if $n \equiv 16 \pmod{24}$,}
\end{cases}
\]
provided $L_\ell\neq \emptyset$ \cite{{Venkov1},{Venkov2}} (see also \cite{Pache}). 
%{\color{blue} TO BE DELETED: 
%In this paper, we call $C_\ell$ (resp.\ $L_\ell$) a shell of the code $C$ 
%(resp.\ lattice $L$)  whenever it is non-empty.
%}

The main results of the present paper is 
to give an analogue of these results for non-extremal codes and lattices,
as follows: 

\begin{thm}\label{thm:AM}
\begin{enumerate}
\item [{\rm (1)}]
Let $C$ be a {binary} doubly even self-dual code of length $24m$ 
with minimum weight $4m$. Then every shell of $C$ is a 
combinatorial $1$-design.

\item [{\rm (2)}]
Let $C$ be a ternary self-dual code of length $12m$ 
with minimum weight $3m$. 
Then for $\ell\leq 6m-3$, 
%every shell of 
$C_\ell$ is a 
combinatorial $1$-design.

\item [{\rm (3)}]

Let $C$ be a quaternary self-dual code of length $6m$ 
with minimum weight $2m$. 
Then for $\ell\leq 3m-1$, 
%every shell of 
$C_\ell$ is a 
combinatorial $1$-design.

\item [{\rm (4)}]
Let $L$ be an even unimodular lattice of rank $24m$ with minimum norm
$2m$. Then every shell of $L$ supports a spherical $3$-design.

\end{enumerate}
\end{thm}

\begin{rem}
Theorem~\ref{thm:AM} (1) and (4) were firstly obtained 
by B.~Venkov and the second named author using modular forms.
We remark that the proof of Theorem \ref{thm:AM} (1)
in the present paper is
different from their original proof. 
\end{rem}

As an application of Theorem~\ref{thm:AM},
we give some restrictions on the 
weight enumerators of self-dual codes. 

\begin{cor}\label{cor:res}
\begin{enumerate}
\item [{\rm (1)}]
Let $C$ be a {binary} doubly even self-dual code of length $24m$ 
with minimum weight $4m$. 
Then the coefficient of $x^{24m-4m}y^{4m}$ 
in the weight enumerator of $C$ 
is divisible by $6$. 

\item [{\rm (2)}]
Let $C$ be a ternary self-dual code of length $12m$ 
with minimum weight $3m$. 
Then the coefficient of $x^{12m-3m}y^{3m}$
in the weight enumerator of $C$ 
is divisible by $4$. 

\item [{\rm (3)}]

Let $C$ be a quaternary self-dual code of length $6m$ 
with minimum weight $2m$. 
Then the coefficient of $x^{6m-2m}y^{2m}$
in the weight enumerator of $C$ 
is divisible by $3$. 
\end{enumerate}
\end{cor}

The following theorem 
gives a strengthening of Theorem~\ref{thm:AM} (1) and (4)
for some particular cases.

\begin{thm}\label{thm:rising}

\begin{enumerate}
\item [{\rm (1)}]
\label{thm:code-2}
Let $C$ be a {binary} doubly even self-dual code of length $96$ 
with minimum weight $16$. Then $C_{20}$ (also $C_{76}$) is a 
combinatorial $2$-design.

\item [{\rm (2)}]
Let $L$ be an even unimodular lattice of rank $240$ with minimum norm
$20$. Then $L_{22}$ supports a spherical $5$-design.

\end{enumerate}
\end{thm}

On the other hand, 
we give an upper bound on the value $t$ of 
combinatorial $t$-designs 
formed by a shell of a {binary} doubly even self-dual code of length $96$ 
with minimum weight $16$. For convenience, let
\begin{equation}\label{cL}
\cL=\{\ell\in\mathbb{Z}\mid 16\leq\ell\leq80,\;\ell\equiv0\pmod{4}\}.
\end{equation}

\begin{thm}\label{thm:ub}
Let $C$ be a {binary} doubly even self-dual code of length $96$ 
with minimum weight $16$, and let $\ell\in \cL$.
Assume that $C_{\ell}$ is a combinatorial $t$-design. 
Then the following statements hold:
\begin{enumerate}
\item [{\rm (1)}]
If $t=2$ and 
$\ell\neq 20,76$,
then every shell of $C$ is a combinatorial $2$-design.
\item [{\rm (2)}]
If $t\geq3$ and 
$\ell\neq 20,48,76$,
then every shell of $C$ is a combinatorial $t$-design.
\item [{\rm (3)}]
We have
\[t\leq\begin{cases}
7&\text{if $\ell=48$,}\\
5&\text{if $\ell=20,76$.}\\
4&\text{otherwise.}
\end{cases}\]
\end{enumerate}
\end{thm}

This paper is organized as follows. 
In Section~\ref{sec:pre}, 
we give definitions and some basic properties of 
self-dual codes, unimodular lattices, 
combinatorial $t$-designs and spherical $t$-designs 
used in this paper.
In Sections~\ref{sec:AM} and \ref{sec:rising}, 
we give proofs of Theorem~\ref{thm:AM} and \ref{thm:rising}, 
respectively.
In Section~\ref{sec:96}, 
we give a proof of Theorem~\ref{thm:ub}, 
give the known examples of 
binary doubly even self-dual codes of length $96$ 
with minimum weight $16$ and 
investigate their designs. 
Finally, in Section~\ref{sec:rem}, 
we give concluding remarks.

All computer calculations in this paper were done with the help of 
Magma~\cite{Magma} and Mathematica~\cite{Mathematica}. 

\section{Preliminaries}\label{sec:pre}

\subsection{Codes and combinatorial $t$-designs}

A linear code $C$ of length $n$ is a linear subspace of $\FF_{q}^{n}$. 
%{
For $q=2$ and $q=3$, 
an inner product $({x},{y})$ on $\FF_q^n$ is given 
by
\[
(x,y)=\sum_{i=1}^nx_iy_i,
\]
where $x,y\in \FF_q^n$ with $x=(x_1,x_2,\ldots, x_n)$ and 
$y=(y_1,y_2,\ldots, y_n)$. 
The Hermitian inner product $({x},{y})$ on $\FF_4^n$ is given 
by
\[
(x,y)_H=\sum_{i=1}^nx_iy_i^2,
\]
where $x,y\in \FF_4^n$ with $x=(x_1,x_2,\ldots, x_n)$ and 
$y=(y_1,y_2,\ldots, y_n)$. 
The dual of a linear code $C$ is defined as follows: 
for $q=2$ and $q=3$, 
\[
C^{\perp}=\{{y}\in \FF_{q}^{n}\mid ({x},{y}) =0 \text{ for all }{x}\in C\},
\]
for $q=4$, 
\[
C^{\perp,H}=\{{y}\in \FF_{q}^{n}\mid ({x},{y})_H =0 \text{ for all }{x}\in C\}.
\]
A linear code $C$ is called self-dual 
if $C=C^{\perp}$ for $q=2$ and $q=3$ and 
if $C=C^{\perp,H}$ for $q=4$. 
For $x \in\FF_q^n$,
the weight $\wt(x)$ is the number of its nonzero components. 
In this paper, we consider the following self-dual codes~\cite{CS}: 
\begin{tabbing}
Doubly even: A code is defined over $\FF_{2}^{n}$ with all weights divisible by $4$,\\
Ternary: A code is defined over $\FF_{3}^{n}$ with all weights divisible by $3$,\\
Quaternary: A code is defined over $\FF_{4}^{n}$ with all weights divisible by $2$. 
\end{tabbing}

A combinatorial $t$-design 
is a pair 
$\mathcal{D}=(\Omega,\mathcal{B})$, where $\Omega$ is a set of points of 
cardinality $v$, and $\mathcal{B}$ a collection of $k$-element subsets
of $\Omega$ called blocks, with the property that any $t$ points are 
contained in precisely $\lambda$ blocks.

The support of a vector 
${x}:=(x_{1}, \dots, x_{n})$, 
$x_{i} \in \GF_{q}$ is 
the set of indices of its nonzero coordinates: 
${\rm supp} ({x}) = \{ i \mid x_{i} \neq 0 \}$\index{$supp (x)$}.
Let $\Omega:=\{1,\ldots,n\}$ and 
$\mathcal{B}(C_\ell):=\{\supp({x})\mid {x}\in C_\ell\}$. 
Then for a code $C$ of length $n$, 
we say that $C_\ell$ is a combinatorial $t$-design if 
$(\Omega,\mathcal{B}(C_\ell))$ is a combinatorial $t$-design.

\subsection{Harmonic weight enumerators}\label{sec:Har}

In this section, we extend the method of 
harmonic weight enumerators 
which were used by Bachoc~\cite{Bachoc} and 
Bannai et al.~\cite{Bannai-Koike-Shinohara-Tagami}.
For the readers convenience we quote from~\cite{Bachoc,Delsarte}
the definitions and properties of discrete harmonic functions 
(for more information the reader is referred to~\cite{Bachoc,Delsarte}).

Let $\Omega=\{1, 2,\ldots,n\}$ be a finite set (which will be the set of coordinates of the code) and 
let $X$ be the set of its subsets, while, for all $k= 0,1,\dots, n$, 
$X_{k}$ is the set of its $k$-subsets.
We denote by $\R X$ and $\R X_k$ the 
real vector spaces spanned by the elements of $X$
and $X_{k}$, respectively.
An element of $\R X_k$ is denoted by
$$f=\sum_{z\in X_k}f(z)z$$
and is identified with the real-valued function on $X_{k}$ given by 
$z \mapsto f(z)$. 

An element $f\in \R X_k$ can be extended to an element $\widetilde{f}\in \R X$ by setting, for all $u \in X$,
$$\widetilde{f}(u)=\sum_{z\in X_k, z\subset u}f(z).$$
If an element $g \in \R X$ is equal to $\widetilde{f}$ 
for some $f \in \R X_{k}$, then we say that $g$ has degree $k$. 
The differentiation $\gamma$ is the operator defined by linearity from 
$$\gamma(z) =\sum_{y\in X_{k-1},y\subset z}y$$
for all $z\in X_k$ and for all $k=0,1, \ldots n$, and $\Harm_{k}$ is the kernel of $\gamma$:
$$\Harm_k =\ker(\gamma|_{\R X_k}).$$

%{\color{blue} TO BE DELETED: 
%\begin{Thm}[{{\cite[Theorem 7]{Delsarte}}}]%\label{thm:design}
%A set $\mathcal{B} \subset X_{m}$, where $m \leq n$, of blocks is a combinatori%al $t$-design 
%if and only if $\sum_{b\in \mathcal{B}}\widetilde{f}(b)=0$ 
%for all $f\in \Harm_k$, $1\leq k\leq t$. 
%\end{Thm}
%}
\begin{lem}[{{\cite[Theorem 7]{Delsarte}},{\cite[Lemma 2.5]{Tanabe}}}]\label{thm:design}
Let $C$ be a code of length $n$ with minimum weight $d$. 
Let $w_0$ be the largest integer satisfying 
\[
w_0 -\left\lfloor\frac{w_0+q-2}{q-1}\right\rfloor <d, 
\]
where, if $q = 2$, we take $w_0 := n$. 
Let $i$ be a weight of $C$ such that $d \leq i \leq w_0$. 
Then the subset of $\{1,2,\ldots,n\}$ 
which support codewords of weight $i$ in $C$ form a $t$-design 
if and only if 
\[
\sum_{u\in C,\wt(u)=i}\widetilde{f}(u)=0,
\]
for all $f\in \Harm_k$, 
$1 \leq  k \leq  t$.
\end{lem}

In \cite{Bachoc}, the harmonic weight enumerator associated to a linear code $C$ was defined as follows: 
\begin{Def}[{\cite[Definition 2.1]{Bachoc}},{\cite[Definition 4.1]{Bachoc2}}]
Let $C$ be a linear code of length $n$
and let $f\in\Harm_{k}$. 
The harmonic weight enumerator associated to $C$ and $f$ is
\[
W_{C,f}(x,y)=\sum_{c\in C}\widetilde{f}(c)x^{n-\wt(c)}y^{\wt(c)}.
\]
\end{Def}

Then 
the submodules of harmonic weight enumerators
are described as follows: 
\begin{Thm}[{\cite[Lemma 3.1]{Bachoc}},{\cite[Lemma 6.1 and 6.2]{Bachoc2}}]\label{thm:invariant}
Let $C$ be a linear code of length $n$,
and let $f \in \Harm_{k}$.
\begin{enumerate}
\item [{\rm (1)}]
Suppose $C$ is a binary doubly even self-dual code.
Then
\begin{align*}
W_{C,f}(x,y)&\in
\begin{cases}
(xy)^k \CC[P_8,P_{24}]&\text{ if }k \equiv 0\pmod{4},\\
(xy)^k P_{12}\CC[P_8,P_{24}]&\text{ if }k \equiv 2\pmod{4},\\
(xy)^k P_{18}\CC[P_8,P_{24}]&\text{ if }k \equiv 3\pmod{4},\\
(xy)^k P_{30}\CC[P_8,P_{24}]&\text{ if }k \equiv 1\pmod{4},
\end{cases}
\intertext{where}
P_{8}&=x^8+14x^4y^4+y^8, \\
P_{12}&=x^2y^2(x^4-y^4)^2, \\
P_{18}&=xy(x^8-y^8)(x^8-34x^4y^4+y^8), \\
P_{24}&=P_{12}^2,\\
P_{30}&=P_{12}P_{18}.
\end{align*}

\item [{\rm (2)}]
Suppose $C$ is a ternary self-dual code.
Then
\begin{align*}
W_{C,f}&\in 
\begin{cases}
xy p_{14}\CC[g_4,g_{12}],&\text{ if }k=1,\\
(xy)^2 p_{4}\CC[g_4,g_{12}],&\text{ if }k=2,\\
\end{cases}
\intertext{where}
p_{4}&=y(x^3-y^3), \\
p_{14}&=y^2(x^3-y^3)^2(x^6-20x^3y^3-8y^6), \\
g_{4}&=x^4+8xy^3, \\
g_{12}&=y^3(x^3-y^3)^3.
\end{align*}

\item [{\rm (3)}]
Suppose $C$ is a quaternary self-dual code.
Then
\begin{align*}
W_{C,f}&\in 
\begin{cases}
xy q_{6}\CC[h_2,h_{6}],&\text{ if }k=1,\\
(xy)^2 \CC[h_2,h_{6}],&\text{ if }k=2,\\
\end{cases}
\intertext{where}
h_{2}&=x^2+3y^2, \\
h_{6}&=y^2(x^2-y^2)^2, \\
q_{6}&=y(x^2-y^2)(x^3-9xy^2).
\end{align*}

\end{enumerate}
\end{Thm}

\begin{rem}
In {\cite[Lemma 3.1]{Bachoc}} and {\cite[Lemma 6.1 and 6.2]{Bachoc2}},
explicit sets of generators of the submodules for general $k$ were given.
We omit listing them here, since we do not need them.
\end{rem}
\subsection{Lattices and spherical $t$-designs}
A lattice in $\R^{n}$ is a subgroup $L \subset \R^{n}$ 
with the property that there exists a basis 
$\{e_{1}, \ldots, e_{n}\}$ of $\R^{n}$ 
such that $L =\Z e_{1}\oplus \cdots \oplus\Z e_{n}$.
The dual lattice of $L$ is the lattice
\[
L^{\sharp}:=\{y\in \R^{n}\mid (y,x) \in\Z , \ \forall x\in L\}, 
\]
where $(x,y)$ is the standard inner product. 
In this paper, we assume that the lattice $L $ is integral, 
that is, $(x,y) \in\Z$ for all $x$, $y\in L$. 
An integral lattice $L$ is called even if $(x,x) \in 2\Z$ for all $x\in L$. 
An integral lattice $L$ is called unimodular if $L^{\sharp}=L$. 

The concept of a spherical $t$-design is due to Delsarte--Goethals--Seidel 
\cite{DGS}. For a positive integer $t$, a finite nonempty set $X$
in the unit sphere
\[
S^{n-1} = \{x = (x_1, \ldots , x_n) \in \R ^{n}\ |\ x_1^{2}+ \cdots + x_n^{2} = 1\}
\]
is called a spherical $t$-design in $S^{n-1}$ if the following condition is satisfied:
\[
\frac{1}{|X|}\sum_{x\in X}f(x)=\frac{1}{|S^{n-1}|}\int_{S^{n-1}}f(x)d\sigma (x), 
\]
for all polynomials $f(x) = f(x_1, \ldots ,x_n)$ of degree not exceeding $t$. 
A finite subset $X$ in $S^{n-1}(r)$, the sphere of radius $r$ 
centered at the origin, 
is also called a spherical $t$-design if $(1/r)X$ is 
a spherical $t$-design in the unit sphere $S^{n-1}$.
Then we say that $L_\ell$ is a spherical $t$-design 
if $(1/\sqrt{\ell})L_\ell$ is a spherical $t$-design. 

Let 
$\Harm_{j}(\R^{n})$ denote
the set of homogeneous 
harmonic polynomials of degree $j$:
\[
\Harm_{j}(\R^{n})=
\{f\in \CC[x_1, \ldots,x_n] \mid \deg(f) = j \text{ and } 
\sum_{i=1}^n\frac{\partial^2}{\partial^2x_i}
f = 0\}. 
\]
It is well known that $X$ is a spherical $t$-design 
if and only if the condition 
\[
\sum_{x\in X}P(x)=0 
\]
holds for all $P \in\Harm_j(\R^n)$ with $1 \leq  j \leq  t$. 
If the set $X$ is antipodal, that is $-X=X$, and $j$ is odd, then the above condition is fulfilled automatically. So we reformulate the condition of spherical $t$-design 
for antipodal sets as follows:

\begin{prop}\label{thm:design-lattice}
A nonempty finite antipodal subset $X\subset S^{n-1}_{m}$ is a spherical 
$(2s+1)$-design if and only if the condition 
\[
\sum_{x\in X}P(x)=0 
\]
holds for all $P \in\Harm_{2j}(\R^n)$ with 
$1\leq j\leq s$.
\end{prop}

\subsection{Spherical theta series}

Let $\mathbb{H} :=\{z\in\CC\mid {\rm Im}(z) >0\}$ be the upper half-plane. 
\begin{df}
Let $L$ be an integral lattice in $\R^{n}$. 
For a polynomial $P$, the function 
on $\mathbb{H}$ defined by
\[
\vartheta _{L, P} (z):=\sum_{x\in L}P(x)
e^{\pi iz(x,x)}
\]
is called the theta series of $L $ weighted by $P$. 
\end{df}

\begin{rem}
$\\ $
\vspace{-10pt}
\begin{itemize}
\item 
[\rm (i)]
When $P=1$, we get the classical theta series 
\[
\vartheta _{L} (z)=\vartheta _{L, 1} (z)=\sum_{m\ge 0}|L_{m}|q^{m},
\quad\text{where }q=e^{\pi i z}. 
\]
\item 
[\rm (ii)] 
The weighted theta series can be written as 
\begin{equation}\label{weighted_th}
\vartheta _{L, P} (z)
=\sum_{m\geq 0}a^{(P)}_{m}q^{m}, 
\text{ where } a^{(P)}_{m}:=\sum_{x\in L_{m}}P(x). 
\end{equation}
\end{itemize}
\end{rem}

\begin{lem}[{\cite[Lemma 5]{Pache}},{\cite{Venkov1},\cite{Venkov2}}]\label{lem:lempache}
Let $L$ be an integral lattice in $\R^{n}$. Then, for $m>0$, 
the non-empty shell $L_{m}$ is a spherical 
$(2s+1)$-design
if and only if 
\[
a^{(P)}_{m}=0\quad\text{for every }P\in\Harm_{2j}(\R^{n}) 
\text{ and $1\leq j\leq s$,}
\]
where $a^{(P)}_{m}$ is the Fourier coefficient of the weighted theta series 
\eqref{weighted_th}.
\end{lem}

For example, we consider an even unimodular lattice $L$. 
Then the weighted theta series $\vartheta_{L, P}(z)$ of $L $ weighted by 
a harmonic polynomial $P$, is a modular form with respect to $SL_{2}(\Z)$. 
In general, we have the following: 
\begin{lem}[{\cite[Theorem 12, Proposition 16]{Pache}}]\label{rem:pache_2.1}
Let $L\subset \RR^n$ be an even unimodular lattice of of 
rank $n = 8N$ and
of minimum norm $2M$.
Let 
\begin{align*}
E_4(z)&:=1+240\sum_{n=1}^{\infty}\sigma_{3}(n)q^{2n}, \\
E_6(z)&:=1-504\sum_{n=1}^{\infty}\sigma_{5}(n)q^{2n}, \\
\Delta(z)&:=\frac{E_4(z)^3-E_6(z)^2}{1728}\\
&=q^2-24q^4+\cdots, 
\end{align*}
where $\sigma_{k-1}(n):=\sum_{d\mid n}d^{k-1}$, and $q=e^{\pi iz}$. 
Then we have 
for $P\in\Harm_{2j} (\RR_n)$, 
\[
\vartheta_{L,P}\in
\begin{cases}
\CC[E_4,\Delta]& \text{if $j$ is even},\\
E_6\CC[E_4,\Delta]& \text{if $j$ is odd}. 
\end{cases}
\]
More precisely, 
there exist $c_i\in \CC$ such that 
\[
\vartheta_{L,P}=
\begin{cases}
\displaystyle\sum_{i=M}^{[(N+j/2)/3]}
c_i\Delta^i
E_4^{N+j/2-3i}& \text{if $j$ is even},\\
\displaystyle
\sum_{i=M}^{[(N+j/2)/3]}
c_iE_6\Delta^i
E_4^{N+(j-3)/2-3i}&\text{if $j$ is odd}. 
\end{cases}
\]
In particular, 
$\vartheta_{L,P} = 0$
if $j$ is even and $3M > N + j/2$, 
or $j$ is odd and $3M > N + (j-3)/2$. 
\end{lem}

%%%%%%%%%%%%%%%%%%%%%%%%%%%%%%%%%%%%%%%%%%%%%%%%%%%%%%%%%%%%%%%%%%%%%%%%%%%%%
\section{Proof of Theorem \ref{thm:AM} and
Corollary \ref{cor:res}}\label{sec:AM}

In this section, 
we give proofs of Theorem \ref{thm:AM} and Corollary~\ref{cor:res}.

\subsection{Proof of Theorem \ref{thm:AM} (1)}

Let $C$ be a binary doubly even self-dual code of 
length $n=24m$, 
and let $f\in\Harm_{1}$.
It is enough to show that $W_{C,f}(x,y)=0$ by Theorem~\ref{thm:design}. 
By Theorem~\ref{thm:invariant} (1)
we have 
\[
W_{C,f}(x,y) =
xy 
 P_{30}Q, 
\]
where $Q\in \CC[P_8,P_{24}]$.
Because of $\min(C)=4m$, 
$W_{C,f}(x,y)$ is divisible by $y^{4m}$. This implies that
$Q$ is divisible by $y^{4m-4}$, and hence
$Q = P_{24}^{m-1}Q'$ for some $Q'\in \CC[P_8,P_{24}]$. 
If this polynomial is nonzero, then 
\[
24m=\deg W_{C,f}(x,y)=32+\deg Q\geq32+24(m-1).
\]
This contradiction proves $Q=0$, and hence
$W_{C,f} (x, y) = 0$.

%%%%%%%%%%%%%%%%%%%%%%%%%%%%%%%%%%%%%%%%%%%%%%%%%%%%%%%%%%%%%%%%%%%%%%%%%%%%%
\subsection{Proof of Theorem \ref{thm:AM} (2)}

Let 
$C$ be a ternary self-dual code of 
length $n=12m$, 
and let $f\in\Harm_{1}$.
It is enough to show that $W_{C,f}(x,y)=0$ by Lemma~\ref{thm:design}. 
By Theorem~\ref{thm:invariant} (2) 
we have 
\[
W_{C,f}(x,y) =xy p_{14}Q,
\]
where $Q\in \CC[g_{4},g_{12}]$.
Because of $\min(C)=3m$, $W_{C,f}(x,y)$ is divisible by $y^{3m}$.
This implies that $Q$ is divisible by $y^{3m-3}$, and hence
$Q = g_{12}^{m-1}Q'$ for some $Q'\in \CC[g_4,g_{12}]$. 
If this polynomial is nonzero, then 
\[
12m=\deg W_{C,f}(x,y)=16+\deg Q\geq 16+12(m-1).
\]
This contradiction proves $Q=0$, and hence
$W_{C,f} (x, y) = 0$. 

%%%%%%%%%%%%%%%%%%%%%%%%%%%%%%%%%%%%%%%%%%%%%%%%%%%%%%%%%%%%%%%%%%%%%%%%%%%%%
\subsection{Proof of Theorem \ref{thm:AM} (3)}

Let 
$C$ be a quaternary self-dual code of 
length $n=6m$, 
and let $f\in\Harm_{1}$.
It is enough to show that $W_{C,f}(x,y)=0$ by Lemma~\ref{thm:design}. 
By Theorem~\ref{thm:invariant} (3) 
we have 
\[
W_{C,f}(x,y) =xy 
q_6Q, 
\]
where $Q\in \CC[h_2,h_{6}]$.
Because of $\min(C)=2m$, 
$W_{C,f}(x,y)$ is divisible by $y^{2m}$.
This implies that $Q$ is divisible by $y^{2m-1}$, and hence
$Q = h_{6}^{m-1}Q'$ for some $Q'\in \CC[h_2,h_{6}]$. 
If this polynomial is nonzero, then 
\[
6m=\deg W_{C,f}(x,y)=8+\deg Q\geq 8+6(m-1).
\]
This contradiction proves $Q=0$, and hence
$W_{C,f} (x, y) = 0$. 

\subsection{Proof of Theorem \ref{thm:AM} (4)}
Let $L$ be an even unimodular lattice of rank $24m$ with minimum norm
$2m$.
Let us assume that 
$P\in\Harm_2(\RR^{24m})$. 
Then by Lemma~\ref{rem:pache_2.1} we have 
$\vartheta_{L,P}(z) =0$. Thus, the result follows from 
Lemma~\ref{lem:lempache}.

\subsection{Proof of Corollary \ref{cor:res}}

The following lemma is easily seen. 

\begin{Lem}[{{\cite[Page 3, Proposition 1.4]{CL}}}]\label{lem: divisible}
Let $\lambda(S)$ be the number of blocks containing a given set $S$
of $s$ points in a combinatorial $t$-$(v,k,\lambda)$ design, where $0\leq s\leq t$. Then
\[
\lambda(S)\binom{k-s}{t-s}
=
\lambda\binom{v-s}{t-s}. 
\]
In particular, the number of blocks is
\[\frac{v(v-1)\cdots(v-t+1)}{k(k-1)\cdots(k-t+1)}\lambda.\]
\end{Lem}

We give the proof of Corollary \ref{cor:res} (1). 
The other cases can be proved similarly. 

Let $C$ be a {binary} doubly even self-dual code of length $24m$ 
with minimum weight $4m$. 
By Theorem~\ref{thm:AM} (1), 
$C_{4m}$ is a combinatorial $1$-design. 
Then by Lemma~\ref{lem: divisible}, 
$|C_{4m}|$ is divisible by $6$. 
This means that the coefficient of $x^{24m-4m}y^{4m}$
in the weight enumerator of $C$ is divisible by $6$. 
This completes the proof of Corollary \ref{cor:res} (1).

%%%%%%%%%%%%%%%%%%%%%%%%%%%%%%%%%%%%%%%%%%%%%%%%%%%%%%%%%%%%%%%%%%%%%%%%%%%%%
\section{Proof of Theorem \ref{thm:rising}}\label{sec:rising}
In this section, we give a proof of Theorem~\ref{thm:rising}. 

\subsection{Proof of Theorem \ref{thm:rising} (1)}

Let $C$ be a binary doubly even self-dual code of 
length $n=24m$,
and let $f\in\Harm_{2}$.
Then by Theorem~\ref{thm:invariant} (1) 
we have 
$W_{C,f}(x,y) =(xy)^{2} P_{12}Q$
for some $Q\in\CC[P_8,P_{24}]$. 
Because of $\min(C)=4m$, 
$W_{C,f}$ is divisible by $y^{4m}$. This implies that 
$Q$ is divisible by $y^{4m-4}$, and hence
$Q$ is divisible by $P_{24}^{m-1}$. Since
$Q$ has degree $24m-16$, this forces $W_{C,f}$ to be a
constant multiple of $(xy)^2 P_8 P_{12}^{2m-1}$.
The coefficient of $y^{4m+4}$ in 
$(xy)^2 P_8 P_{12}^{2m-1}$ is equal to the coefficient
of $y^4$ in
\begin{equation}\label{P8P12}
(x^8+14x^4y^4+y^8)(x^4-y^4)^{4m-2}
\end{equation}
which is $16-4m$. This vanishes
when $m=4$. Therefore, $C_{4m+4}=C_{20}$ is a $2$-design
by Lemma~\ref{thm:design}.

%%%%%%%%%%%%%%%%%%%%%%%%%%%%%%%%%%%%%%%%%%%%%%%%%%%%%%%%%%%%%%%%%%

\subsection{Proof of Theorem \ref{thm:rising} (2)}
Let $L$ be an even unimodular lattice of rank $24m$ with minimum norm
$2m$.
Let us assume that 
$P\in\Harm_4(\RR^{24m})$. 
Then by the Lemma~\ref{rem:pache_2.1} 
we have $\vartheta_{L,P}(z)\in\CC[E_4,\Delta]$.
Since $L$ has minimum norm $2m$, $\vartheta_{L,P}(z)$
is a constant multiple of $\Delta^m E_4$. The coefficient
of $q^{2m+2}$ in $\Delta^m E_4$ is equal to the
coefficient of $q^2$ in 
\[(1-24q^2)^m(1+240q^2),\]
which is $240-24m$. This vanishes when $m=10$. Therefore,
$L_{22}$ supports a spherical $5$-design by 
Lemma~\ref{lem:lempache}.

%%%%%%%%%%%%%%%%%%%%%%%%%%%%%%%%%%%%%%%%%%%%%%%%%%%%%%%%%%%%%%%%%%%%%%%%%%%%%
\section{Binary doubly even self-dual codes of length 96}\label{sec:96}

In order to prove Theorem~\ref{thm:ub},
let us introduce the
fundamental equation for combinatorial designs in \cite{Koch}. 
Let $\Omega=\{1,\dots,n\}$, and suppose that $(\Omega,\mathcal{B})$
is a combinatorial $t$-$(n,k,\lambda)$ design.
Let $c' \in \FF_2^n$ and
\[u_j (c') := |\{x \in 
\mathcal{B}\mid |\supp(c') \cap x| = j\}|. \]
Then the following holds:
\begin{align}\label{eqn:fund}
\sum_{j=\mu}^{\wt(c')}\binom{j}{\mu}u_j(c')
=\lambda_{\mu}
\binom{\wt(c')}{\mu}
,\quad \mu = 1,\ldots, t, 
\end{align}
where we denote by $\lambda_\mu$ the
number of blocks which contain a given set of $\mu$ points. 
Here we note that
\[
\lambda_\mu =
\frac{k(k - 1)\cdots(k - \mu + 1)}
{n(n - 1)\cdots(n - \mu + 1)}
|\mathcal{B}|
\]
by Lemma~\ref{lem: divisible}. Another consequence of
Lemma~\ref{lem: divisible} is the following.

\begin{lem}\label{lem:divCl}
If $(\Omega,\mathcal{B})$
is a combinatorial $t$-$(n,k,\lambda)$ design, then 
$|\mathcal{B}|$ is divisible by the numerator of
\[\frac{n(n-1)\cdots(n-s+1)}{k(k-1)\cdots(k-s+1)}
\quad(s=1,\dots,t)\]
as an irreducible fraction.
\end{lem}

For the remainder of this section,
we let $C$ be a {binary} doubly even self-dual code of length $96$ 
with minimum weight $16$. 
In \cite{DGH1997}, the weight enumerator of $C$ is
determined to be
\begin{align}\label{eq:wt enu}
&W_{C}(x,y)
\\&=1 + (-28086 + a) y^{16} + (3666432 - 16 a) y^{20} \notag\\ \notag
&+ (366474560 + 120 a) y^{24} + (18658567680 - 560 a) y^{28} \\ \notag
&+ (422018863695 +   1820 a) y^{32} + (4552989336064 -  4368 a) y^{36} \\ \notag
&+ (24292464652992 +     8008 a) y^{40} + (65727332943360 -  11440 a) y^{44}\\ \notag
 &+ (91447307757260 + 12870 a) y^{48} + \cdots, \notag
\end{align}
where 
\begin{equation}\label{eq:abd}
28086 < a \leq 229152.
\end{equation}
Since $C$ has minimum weight $16$, $W_{C,f}(x,y)$ is divisible by $y^{16}$
for $f\in\Harm_k$ with $k\geq1$. 
By Theorem~\ref{thm:invariant} (1), this implies
\begin{equation}\label{Wcf}
W_{C,f}(x,y)\in
\begin{cases}
\CC x^2y^2P_8P_{12}^7 &\text{ if } f\in\Harm_2,\\
\CC x^3y^3P_{12}^6P_{18} &\text{ if } f\in\Harm_3,\\
\CC x^4y^4P_8^2P_{12}^6 &\text{ if } f\in\Harm_4,\\
\CC x^5y^5P_8P_{12}^5P_{18} &\text{ if } f\in\Harm_5. 
\end{cases}
\end{equation}
We first investigate the coefficients of the polynomials
appearing above. 
Note that the coefficient of $x^{96-\ell}y^\ell$
in the harmonic weight enumerator \eqref{Wcf} 
is $0$ unless $\ell\in\cL$, where the set $\cL$ is defined in \eqref{cL}.

\begin{lem}\label{lem:coef}
\begin{enumerate}
\item[{\rm (1)}] 
Let $x^2y^2P_8P_{12}^7=\sum c_{\ell}x^{96-\ell}y^\ell$. 
Then $c_{\ell}\neq 0$ for 
$\ell\in\cL\setminus\{20,76\}$.

\item[{\rm (2)}] 
Let $x^3y^3P_{12}^6P_{18}=\sum c_{\ell}x^{96-\ell}y^\ell$. 
Then $c_{\ell}\neq 0$ for 
$\ell\in\cL\setminus\{48\}$.

\item[{\rm (3)}] 
Let $x^4y^4P_8^2P_{12}^6=\sum c_{\ell}x^{96-\ell}y^\ell$. 
Then $c_{\ell}\neq 0$ for 
$\ell\in\cL$.

\item[{\rm (4)}] 
Let $x^5y^5P_8P_{12}^5P_{18}=\sum c_{\ell}x^{96-\ell}y^\ell$. 
Then $c_{\ell}\neq 0$ for 
$\ell\in\cL\setminus\{48\}$.

\end{enumerate}

\end{lem}

\begin{proof}
We can check directly that ${c_{\ell}}\neq 0$ for $\ell$
in the range specified in (1)--(4).
\end{proof}

\begin{proof}[Proof of Theorem~\ref{thm:ub}]
(1) 
By Lemma~\ref{lem:coef}~(1) and \eqref{Wcf}, we have
$W_{C,f}(x,y)=0$ for $f\in\Harm_2$. 
Since every shell of $C$ is a combinatorial $1$-design by
Theorem~\ref{thm:AM}~(1), it is also a combinatorial $2$-design
by Lemma~\ref{thm:design}.

(2) By Lemma~\ref{lem:coef}~(2) and \eqref{Wcf}, we have
$W_{C,f}(x,y)=0$ for $f\in\Harm_3$. 
Thus, every shell of $C$ is a combinatorial $3$-design
by Lemma~\ref{thm:design}. Similarly, if $t=4,5$,
every shell of $C$ is a combinatorial $t$-design.
The proof is complete if we show $t\leq4$ which, at the same time
proves the inequality in (3) for the case $\ell\neq20,48,76$. 
If $t\geq5$, then, in particular,
$C_{16}$ is a combinatorial 
$5$-design.
Since
\[\frac{96\cdot95\cdot94\cdot93\cdot92}{16\cdot15\cdot14\cdot13\cdot12}
=\frac{2\cdot19\cdot23\cdot31\cdot47}{7\cdot13},\]
Lemma~\ref{lem:divCl} implies that $|C_{16}|=a-28086$ is divisible by
$2\cdot19\cdot23\cdot31\cdot47=1273418$.
Then we have $a \geq 1301504$, contrary to \eqref{eq:abd}.

(3) First, we give the proof for the case $\ell=48$.
Suppose $t\geq8$.
We use the fundamental equation \eqref{eqn:fund}.
Take $c' \in C_{16}$ and write
$x_j :=u_j (c')$
for simplicity. 
Note that $x_j = 0$ if $j > 16$ or $j = $ odd. 
Note also that $|C_{48}|=91447307757260 + 12870 a$. 
Solving the system of equations
\[
\binom{16}{i}\lambda_i
=\sum_{j=0}^{8}\binom{2j}{i}x_{2j}
\quad(i=0,1,\dots,8),
\]
we obtain the solution $x_0 = 8112261172015/13528 + 70785 a/838736$,
and $x_0\not\in\ZZ$ 
for all integers $a\in\Z$ in the range \eqref{eq:abd}.
This is a contradiction. 

Next, we give the proof for the cases 
$\ell=20,76$.
Suppose $t\geq6$.
Then $C_{20}$ is a combinatorial $6$-design.
Since
\[\frac{96\cdot95\cdot94\cdot93\cdot92\cdot91}{20\cdot19\cdot18\cdot17\cdot16\cdot15}
=\frac{2\cdot7\cdot13\cdot23\cdot31\cdot47}{3\cdot5\cdot17},\]
Lemma~\ref{lem:divCl} implies that $|C_{20}|$ is divisible by
$2\cdot7\cdot13\cdot23\cdot31\cdot47=6099002$. Since
$6099002>3666432-16a=|C_{20}|$, this is 
a contradiction.

The inequality for the case $\ell\neq20,48,76$ has already been proved in part (2).
\end{proof}

Recall that $C$ is a {binary} doubly even self-dual code of length $96$ 
with minimum weight $16$.
For each $\ell\in\cL$, we denote by $t_\ell(C)$ the largest integer $t$
such that $C_\ell$ is a combinatorial $t$-design.
Let 
\begin{align*} 
\delta(C)&=\min\{t_\ell(C)\mid\ell\in\cL\},\\
s(C)&=\max\{t_\ell(C)\mid\ell\in\cL\}. 
\end{align*} 
In \cite{extremal design2 M-N, support design triply even code 48 M-N}, 
the first and third authors
considered the possible occurrence of $\delta(C)<s(C)$.
By Theorems \ref{thm:AM} (1) and \ref{thm:ub}, we have
\[1\leq\delta(C) \leq 4,\ 2 \leq s(C) \leq7.\]
We will show in Example~\ref{ex:known code} that, for all the known codes $C$,
$\delta(C)=1<2=s(C)$ holds. More precisely,
\begin{equation}\label{tlCmin}
t_\ell(C)=\begin{cases}
2&\text{if $\ell\in\{20,76\}$,}\\
1&\text{if $\ell\in\cL\setminus\{20,76\}$.}
\end{cases}
\end{equation}
To do this, let us begin with the
following lemma: 

\begin{lem}\label{lem:tlC}
\begin{enumerate}
\item[{\rm (1)}]
If $t_{20}(C)\geq3$, then
$a=141k+28086$ for some integer $k$ with $1 \leq k \leq 1426$.
\item[{\rm (2)}]
If $t_\ell(C)\geq2$ 
for some $\ell\in \cL\setminus\{20,76\}$, then 
$a=114k+28086$ for some integer $k$ with $1 \leq k \leq 1763$.
\end{enumerate}
\end{lem}
\begin{proof}
(1) 
Since
\[\frac{96\cdot95\cdot94}{20\cdot19\cdot18}
=\frac{8\cdot47}{3},\quad
\frac{96\cdot95}{20\cdot19}=8\cdot3,
\]
Lemma~\ref{lem:divCl} implies that $|C_{20}|=3666432-16a$ is divisible by
$8\cdot3\cdot47=1128$. Thus $a\equiv27\pmod{141}$.
The result then follows from \eqref{eq:abd}.

(2)
By Theorem~\ref{thm:ub} (1), 
we have $t_{16}(C)\geq2$.
Since
\[\frac{96\cdot95}{16\cdot15}=2\cdot19,\quad
\frac{96}{16}=6,\]
Lemma~\ref{lem:divCl} implies that $|C_{16}|=a-28086$ is divisible by
$2\cdot3\cdot19=114$.
The result then follows from \eqref{eq:abd}.
\end{proof}

The known examples of a {binary} 
doubly even self-dual code of length $96$ with minimum weight $16$ are 
as follows: 
\begin{Ex}\label{ex:known code}
\begin{enumerate}
\item[(a)] Feit \cite{Feit}: $a=37722$.
\item[(b)] Dougherty, Gulliver and Harada \cite{DGH1997}: $a\in\{ 37584, 37500, 37524, 37596 \}$.
\item[(c)] Dontcheva \cite{Dontcheva}: $a\in\{36918, 37884,37332 \}$.
\item[(d)] Harada, Kiermaier, Wassermann and Yorgova \cite{HKWY}: $a=37194$.
\item[(e)] Gulliver and Harada \cite[Table 5, 6]{GH2019}: there are 639 values of $a$.
\end{enumerate}
\end{Ex}

It is claimed in 
\cite{DGH1997} 
that there exists a code with $a=37598$. 
However, this value is incorrect, as it contradicts Corollary~\ref{cor:res}.
We verified that the correct value is $a=37596$.

If $C$ is one of the
codes in 
Example~\ref{ex:known code} (a)--(d), 
Lemma~\ref{lem:tlC} implies immediately
that \eqref{tlCmin} holds.

For some codes in Example~\ref{ex:known code} (e),
Lemma~\ref{lem:tlC}~(1) does not rule out larger values of $t_{20}(C)$.
In fact, all the 
1532 bordered double circulant
codes in \cite[Table~6]{GH2019}, and 
117 codes out of 4565 in \cite[Table~5]{GH2019}
satisfy the condition of Lemma~\ref{lem:tlC}~(1). We have checked, however,
by Magma that none of these codes $C$ satisfies $t_{20}(C)\geq3$.

Some codes given in \cite{GH2019} satisfies the condition of
Lemma~\ref{lem:tlC}~(2). More precisely, 
92 bordered double circulant codes in \cite[Table~6]{GH2019}, 
and 246 codes out of 4565 in \cite[Table~5]{GH2019}
satisfy the condition of Lemma~\ref{lem:tlC}~(2). We have checked, however,
by Magma that none of these codes $C$ satisfies $t_{16}(C)\geq2$.
Therefore, for codes in (e), \eqref{tlCmin} holds as well.

We note that 
there is no known example $C$ with $C_{20}=\emptyset$.

%%%%%%%%%%%%%%%%%%%%%%%%%%%%%%%%%%%%%%%%%%%%%%%%%%%%%%%%%%%%%%%%%%%%%%%%%%%%%
\section{Concluding remarks}\label{sec:rem}

\begin{rem}\label{rem:nonexistence}
A code or lattice satisfying the 
assumption 
of Theorem~\ref{thm:AM} 
is called near-extremal.
In \cite{MOS75}, 
it is shown 
that for sufficiently large $n$, 
there is no extremal or 
near-extremal codes (resp.\ lattices) 
with length $n$ (resp.\ rank $n$).

More precisely, 
it is shown in 
\cite{Zhang(1999)} that 
there is no extremal code with length $n$ for
\begin{center}
\begin{tabular}{rl}
Doubly even:&$n=24i\ (i\geq 154)$, $24i+8\ (i\geq 159)$, $24i+16\ (i\geq 164)$,\\
Ternary:&$n=12i\ (i\geq 70)$, $12i+4\ (i\geq 75)$, $12i+8\ (i\geq 78)$,\\
Quaternary:&$n=6i\ (i\geq 17)$, $6i+2\ (i\geq 20)$, $6i+4\ (i\geq 22)$,
\end{tabular}
\end{center}
and it is shown in 
\cite{HK} that 
there is no near-extremal code with length $n$ for
\begin{center}
\begin{tabular}{rl}
Doubly even:& $n = 24i\ (i \geq 315)$, $24i + 8\ (i \geq 320)$, $24i + 16\ (i \geq 325)$,\\
Ternary:& $n = 12i\ (i \geq 147)$, $12i + 4\ (i \geq 150)$, $12i + 8\ (i \geq 154)$,\\
Quaternary:& $n = 6i\ (i \geq 38)$, $6i + 2\ (i \geq 41)$, $6i + 4\ (i \geq 43)$.
\end{tabular}
\end{center}
Moreover, it is shown in 
\cite{JR} that 
there is no extremal lattice with rank $n>163264$. 
We do not know whether it can be proved that 
there is no near-extremal lattice with sufficiently large rank.
\end{rem}

\begin{rem}
Our proof of Theorem~\ref{thm:rising} actually shows that,
$m=4$ is the only case where we can show that $t_{4m+4}(C)\geq2$.
In fact, it can be easily verified by computer that
nontrivial vanishing of coefficients of the polynomial
\eqref{P8P12} for $1\leq m\leq 314$ (see Remark~\ref{rem:nonexistence})
occurs only for the case mentioned in our proof.
\end{rem}

\begin{rem}
Ternary and quaternary analogues 
of Theorem~\ref{thm:rising} do not exist. 
Let $C$ be a ternary (resp.\ quaternary) self-dual
code of length $12m$ (resp.\ $6m$) with 
minimum weight $3m$ (resp.\ $2m$). 
By Theorem~\ref{thm:invariant}, for $f\in\Harm_2$, 
the harmonic weight enumerator is written as follows: 
\[
W_{C,f}\in
\begin{cases}
\CC x^2y^2 p_4 g_4 g_{12}^{m-1} &\text{if $C$ is ternary},\\
\CC x^2y^2 h_2 h_6^{m-1}&\text{if $C$ is quaternary}.
\end{cases}
\]
By Lemma \ref{thm:design}, 
if the coefficient of $x^{12m-3\ell}y^{3\ell}$ 
(resp.~$x^{6m-2\ell}y^{2\ell}$)
in 
$W_{C,f}$ is zero for 
$3m\leq 3\ell\leq12m$
(resp.\ $2m\leq 2\ell\leq6m$), 
then $C_{3\ell}$
(resp.\ $C_{2\ell}$)
is a combinatorial $2$-design. However,
no such coefficient vanishes.
\end{rem}

\begin{rem}
An analogue of Theorem~\ref{thm:ub}, parts (1) and (2) seems
to hold. 
That is, for an even unimodular lattice $L$ of rank $240$,
and $t=5,7$, 
if $L_{2\ell}$ supports a spherical $t$-design for some $\ell\neq11$,
then $L_{2\ell}$ supports a spherical $t$-design for all nonzero $\ell$ with
$L_{2\ell}\neq\emptyset$.
The proof would require non-vanishing of coefficients of the modular
forms $\Delta^{10}E_4$ and $\Delta^{10}E_6$ as formal power series.
\end{rem}

\begin{rem}
The case $m=1$ in Theorem \ref{thm:AM} (4) 
was essentially used in the proof of 
the classification of 
even unimodular lattices of rank $24$ \cite{Venkov0} 
(see also \cite[Proposition 3.3 and Corollary 3.5]{Ebeling}). 
\end{rem}

\section*{Acknowledgments}
The authors are supported by JSPS KAKENHI (17K05155, 18K03217). 

%%%%%%%%%%%%%%%%%%%%%%%%%%%%%%%%%%%%%%%%%%%%%%%%%%%%%%%%%%%%%%%%%%%%%%%%%%%%%%%%

\end{document}